\documentclass[11pt,oneside]{amsart}
\usepackage[top=25mm, bottom=25mm, left=25mm, right=25mm]{geometry}
\usepackage{amsfonts,amssymb,amsmath,amsthm,amsxtra}
\usepackage[T1]{fontenc}
\usepackage[utf8]{inputenc}
\usepackage{bm}
\usepackage{mathtools}
\usepackage{dsfont}
\usepackage{bbm}
\usepackage{mathrsfs}
\usepackage{enumitem}

\usepackage{graphics}

\numberwithin{equation}{section}
\newtheorem{theorem}[equation]{Theorem}

\newtheorem{lemma}[equation]{Lemma}

\theoremstyle{definition}

\newcommand{\CC}{\mathbb{C}}
\newcommand{\DD}{\mathbb{D}}
\newcommand{\EE}{\mathbb{E}}

\newcommand{\GG}{\mathbb{G}}

\newcommand{\II}{\mathbb{I}}

\newcommand{\NN}{\mathbb{N}}

\newcommand{\QQ}{\mathbb{Q}}
\newcommand{\RR}{\mathbb{R}}

\newcommand{\TT}{\mathbb{T}}

\newcommand{\ZZ}{\mathbb{Z}}

\newcommand{\C}{\mathbb{C}}
\newcommand{\N}{\mathbb{N}}

\newcommand{\calP}{\mathcal{P}}
\newcommand{\calF}{\mathcal{F}}
\newcommand{\calB}{\mathcal{B}}

\newcommand{\calH}{\mathcal{H}}

\newcommand{\calT}{\mathcal{T}}

\newcommand{\calS}{\mathcal{S}}
\newcommand{\calD}{\mathcal{D}}

\newcommand{\calL}{\mathcal{L}}

\newcommand{\dif}{\mathrm{d}}
\newcommand{\ex}{\bm{e}}

\newcommand*{\DMO}[1]{\expandafter\DeclareMathOperator\csname #1\endcsname {#1}}
\DMO{ord}
\DMO{supp}
\DMO{Sym}
\DMO{cl}
\DMO{lcm}
\DMO{dist}
\DMO{tr}
\DMO{diam}

\DeclarePairedDelimiter\abs{\lvert}{\rvert}

\DeclarePairedDelimiter\norm{\lVert}{\rVert}

\DeclarePairedDelimiter\ceil{\lceil}{\rceil}
\DeclarePairedDelimiterX\spr[2]{\langle}{\rangle}{#1,#2}

\DeclarePairedDelimiterX\Set[2]{\{}{\}}{#1\colon #2}
\DeclarePairedDelimiterX\Seq[1]{(}{)}{#1}
\newcommand\blfootnote[1]{%
  \begingroup
  \renewcommand\thefootnote{}\footnote{#1}%
  \addtocounter{footnote}{-1}%
  \endgroup
}

\begin{document}
\title[Oscillation for singular Radon operators]
	{Oscillation estimates for truncated singular Radon operators}

\author{Wojciech S{\l}omian}
\address[Wojciech S{\l}omian]{
	  Faculty of Pure and Applied Mathematics, 
	  Wroc\l{}aw University of Science and Technology\\
	  Wyb{.} Wyspia\'nskiego 27,
	  50-370 Wroc\l{}aw, Poland}
    \email{wojciech.slomian@pwr.edu.pl}
\thanks{The author was
partially supported by the National Science Centre in Poland, grant
Opus 2018/31/B/ST1/00204.}

\begin{abstract}
In this paper we prove uniform oscillation estimates on $L^p$, with $p\in(1,\infty)$, for truncated singular integrals of the Radon type associated with the Calder\'on--Zygmund kernel, both in continuous and discrete settings. In the discrete case we use the Ionescu--Wainger multiplier theorem and the Rademacher--Menshov inequality to handle the number-theoretic nature of the discrete singular integral.  The result we obtained in the continuous setting can be seen as a generalisation of the results of Campbell, Jones, Reinhold and Wierdl for the continuous singular integrals of the Calder\'on--Zygmund type.
\end{abstract}
\maketitle
		\blfootnote{
The paper is a part of author's doctoral thesis written under the supervision of Professor Mariusz Mirek.
}
\tableofcontents
\section{Introduction}
\label{sec:1}
The purpose of this paper is to give a proof of oscillation inequalities for truncated singular integrals of the Radon type, both in continuous and discrete settings.

Let $k\geq1$ be a natural number and let $\Gamma$ be a fixed finite subset of $\N_0^k\setminus\{0\}$ equipped with the lexicographic order. The space of tuples of real numbers and integer numbers labeled by multi-indices $\gamma=(\gamma_1,\dots,\gamma_k)\in\Gamma$ will be denoted by $\RR^\Gamma$ and $\ZZ^\Gamma$, respectively. Clearly, one has $\RR^\Gamma\cong\RR^{|\Gamma|}$ and $\ZZ^\Gamma\cong\ZZ^{|\Gamma|}$. The canonical polynomial mapping associated with the set $\Gamma$ is defined by
\begin{equation*}
    \RR^k\ni x=(x_1,\dots,x_k)\mapsto(x)^\Gamma:=(x^\gamma\colon\gamma\in\Gamma)\in\RR^\Gamma,
\end{equation*}
where $x^\gamma=x_1^{\gamma_1}x_2^{\gamma_2}\cdots x_k^{\gamma_k}$. The symbol $B(x, t)$ stands for an open Euclidean ball in $\RR^k$ with radius $t>0$ centered at a point $x\in\RR^k$. Let $\Omega\subset\RR^k$ be a bounded convex open set such that $B(0,c_{\Omega}) \subseteq \Omega \subseteq B(0,1)$ for some $c_{\Omega}\in(0, 1)$. For $t>0$, the dilates of the set $\Omega$ are defined as 
$$\Omega_t := \{x\in\RR^{k}: t^{-1}x\in\Omega\}.$$

Let $K\colon\RR^{k}\setminus\{0\} \to \CC$ be the Calder\'on--Zygmund kernel satisfying the following conditions.
\begin{enumerate}
\item The size condition. For every $x\in\RR^k\setminus\{0\}$, we have
\begin{equation}
\label{eq:size-unif}
\abs{K(x)} \leq \abs{x}^{-k}.
\end{equation}
\item The cancellation condition. For every $0<r<R<\infty$, we have
\begin{equation}
\label{eq:cancel}
\int_{\Omega_{R}\setminus \Omega_{r}} K(y) \dif y = 0.
\end{equation}
\item
  The H\"older continuity condition.
  For some $\sigma\in(0, 1]$ and every $x, y\in\RR^k\setminus\{0\}$ with $|y|\leq |x|/2$, we have
\begin{equation}
\label{eq:K-modulus-cont}
\abs{K(x)-K(x+y)}
\leq
\abs{y}^{\sigma} \abs{x}^{-k-\sigma}.
\end{equation}
\end{enumerate}

Now, for finitely supported functions $f\colon\ZZ^{\Gamma}\to \CC$ and $t>0$, we may define the discrete truncated Radon singular operator by setting
\begin{align}
H_t f(x):= \sum_{y \in\Omega_{t}\cap\ZZ^{k}\setminus\{0\}} f(x-(y)^\Gamma)K(y),\quad x\in\ZZ^\Gamma,
\end{align}
where $K$ is the Calder\'on--Zygmund kernel which satisfies the conditions listed above. The operator $H_t$ is a discrete counterpart of the continuous singular Radon transform. Recall that for a given smooth compactly supported function $f\colon\RR^\Gamma\to\CC$ the continuous singular Radon transform of $f$ is defined as
\begin{equation}\label{histRadoncon1}
    \calH_t f(x):={\rm p.v.}\int_{\Omega_t}f(x-(y)^\Gamma)K(y){\rm d}y,\quad x\in\RR^\Gamma.
\end{equation}
\indent Let $\II\subseteq\RR$. For a given increasing sequence $I=(I_j: j\in\NN)\subseteq\II$, the truncated
oscillation seminorm is defined for any family of measurable
functions
$(\mathfrak a_t(x): t>0)\subseteq\CC$ by setting
\begin{align}
\label{eq:45}
O_{I, N}^2(\mathfrak a_t(x): t\in \II)
= \Big(\sum_{j=1}^N\sup_{\substack{I_j \le t < I_{j+1}\\t\in\II}}
\abs{\mathfrak a_t(x)-\mathfrak a_{I_j}(x)}^2\Big)^{1/2}
\quad \text{for all} \quad N\in\NN\cup\{\infty\}.
\end{align}
Our result in the case of the discrete singular operator reads as follows.
\begin{theorem}\label{Thm1}
Let $k\ge 1$ and let consider $\Gamma\subset\N_0^k\setminus\{0\}$ such that $|\Gamma|<\infty$. Then for any $p\in(1,\infty)$ there is a constant $C_{p,k,|\Gamma|}>0$ such that
\begin{equation}\label{oscileqcanon}
\sup_{N\in\NN}\sup_{I\in\mathfrak{S}_N(\RR_+)}\norm[\big]{O_{I,N}^2(H_t f:t\in\RR_+)}_{\ell^p(\ZZ^\Gamma)}\le C_{p,k,|\Gamma|}\norm{f}_{\ell^p(\ZZ^\Gamma)},\quad f\in \ell^p(\ZZ^\Gamma),
\end{equation}
where $\mathfrak{S}_N(\RR_+)$ is the set of all strictly increasing sequences of length $N+1$ contained in $\RR_+$ (see Section~\ref{sec:notation}).
\end{theorem}
\noindent A similar result holds in the continuous case.
\begin{theorem}\label{Thm2}
Let $k\ge 1$ and let consider $\Gamma\subset\N_0^k\setminus\{0\}$ such that $|\Gamma|<\infty$. Then for any $p\in(1,\infty)$ there is a constant $C_{p,k,|\Gamma|}>0$ such that
\begin{equation}
\sup_{N\in\NN}\sup_{I\in\mathfrak{S}_N(\RR_+)}\norm[\big]{O_{I,N}^2(\calH_t f:t\in\RR_+)}_{L^p(\RR^\Gamma)}\le C_{p,k,|\Gamma|}\norm{f}_{L^p(\RR^\Gamma)},\quad f\in L^p(\RR^\Gamma).
\end{equation}
\end{theorem}
\noindent It is worth to note that Theorem~\ref{Thm1} does not follow directly from Theorem~\ref{Thm2} since the discrete nature and number-theoretic complications related to the mapping $(x)^\Gamma$ prevents the use of the standard transference arguments.

By a standard lifting argument (see \cite[Lemma 2.1]{MST1} and \cite[Section 11, pp. 483--485]{bigs}) Theorems~\ref{Thm1} and \ref{Thm2} imply more general results. Namely, let $d,k\ge 1$ be fixed natural numbers and consider the mapping
\begin{equation}\label{polymap}
    \mathcal{P}=(\mathcal{P}_1,\dots,\mathcal{P}_{d})\colon\ZZ^k\to\ZZ^{d},
\end{equation}
where each $\calP_j\colon\ZZ^k\to\ZZ$ is a polynomial with integer coefficients such that $\calP_j(0)=0$. For an appropriate test function $f\colon\ZZ^d\to\C$ (or $f\colon\RR^d\to\C$) we define Radon operators associated to the mapping $\mathcal{P}$ by setting
\begin{align}
    H_t^\calP f(x)&:=\sum_{y \in\Omega_{t}\cap\ZZ^{k}\setminus\{0\}} f(x-\calP(y))K(y),\quad x\in\ZZ^d,\label{eq:1}\\
    \calH_t^\calP f(x)&:={\rm p.v.}\int_{\Omega_t}f(x-\calP(y))K(y){\rm d}y,\quad x\in\RR^d\label{continousRadon}.
\end{align}
The above results may be used to easily deduce the following.
\begin{theorem}\label{Thm3}
Let $d, k\ge 1$ and let $\calP$ be a polynomial mapping~\eqref{polymap}. For any $p\in(1,\infty)$ there is a constant $C_{p,d,k,{\rm deg}\calP}>0$ such that
\begin{align}
    \sup_{N\in\NN}\sup_{I\in\mathfrak{S}_N(\RR_+)}\norm[\big]{O_{I,N}^2(H_t^\calP f:t\in\RR_+)}_{\ell^p(\ZZ^d)}&\le C_{p,d,k,{\rm deg} \calP}\norm{f}_{\ell^p(\ZZ^d)},\quad f\in \ell^p(\ZZ^d),\label{eq:3}\\
    \sup_{N\in\NN}\sup_{I\in\mathfrak{S}_N(\RR_+)}\norm[\big]{O_{I,N}^2(\calH_t^\calP f:t\in\RR_+)}_{L^p(\RR^d)}&\le C_{p,d,k,{\rm deg}\calP}\norm{f}_{L^p(\RR^d)},\quad f\in L^p(\RR^d).\label{coscil}
\end{align}
In particular, the implied constants in the inequalities above are independent of the coefficients of the polynomial mapping $\calP$.
\end{theorem}
\noindent By the properties of the oscillation seminorm inequalities \eqref{eq:3} and \eqref{coscil} implies the maximal estimates
\begin{align*}
    \norm[\big]{\sup_{t>0}|H_t^\calP f|}_{\ell^p(\ZZ^d)}\lesssim\|f\|_{\ell^p(\ZZ^d)} \quad\text{and}\quad \norm[\big]{\sup_{t>0}|\calH_t^\calP f|}_{L^p(\RR^d)}\lesssim\|f\|_{L^p(\RR^d)}.
\end{align*}
Furthermore, from the oscillation inequality one may deduce the pointwise convergence. Namely, the inequality \eqref{coscil} implies that
\begin{equation*}
    \lim_{t\to\infty}\calH_t^\calP f(x) \qquad \text{and}\qquad \lim_{t\to 0}\calH_t^\calP f(x) 
\end{equation*}
exists for almost all $x\in\RR^d$, every $f\in L^p(\RR^d)$ and every $p\in(1,\infty)$.\\
\indent It turns out that the inequality \eqref{eq:3} has an ergodic theoretical interpretation. Let $(X,\calB,\mu)$ be a  $\sigma$-finite measure space endowed with a family of invertible commuting and measure preserving transformations $S_1,\ldots, S_d$. Let us define the following singular operator
\begin{equation}\label{ergodic}
    \calT_t^\calP f(x)=\sum_{y\in\Omega_t\cap\ZZ^k\setminus\{0\}}f(S_1^{\mathcal{P}_1(y)}\circ S_2^{\calP_2(y)}\circ\cdots\circ S_d^{\mathcal{P}_d(y)}x)K(y).
\end{equation}
Then by appealing to the Calderón transference principle  \cite{Cald} we see that by using \eqref{eq:3} one can deduce the following result.
\begin{theorem}
Let $d, k\ge 1$ and let $\calT_t^\calP$ be defined as in \eqref{ergodic}. Then for any $p\in(1,\infty)$ there is a constant $C_{p,d,k,{\rm deg}\calP}>0$ such that
\begin{equation}\label{ergodicosc}
    \sup_{N\in\NN}\sup_{I\in\mathfrak{S}_N(\RR_+)}\norm[\big]{O_{I,N}^2(\calT_t^\calP f:t\in\RR_+)}_{L^p(X)}\le C_{p,d,k,{\rm deg}\calP}\norm{f}_{L^p(X)},\quad f\in L^p(X,\mu).
\end{equation}
The implied constant is independent of the
coefficients of the polynomial mapping $\calP$.
\end{theorem}
\noindent Again, the inequality \eqref{ergodicosc} implies the maximal estimate
\begin{equation*}
    \norm[\big]{\sup_{t>0}|\calT_t^\calP f|}_{L^p(X)}\lesssim \|f\|_{L^p(X)}
\end{equation*}
and pointwise convergence, that is
\begin{equation*}
    \lim_{t\to\infty} \calT_t^\calP(x)
\end{equation*}
exists $\mu$-almost everywhere on $X$ for every $f\in L^p(X,\mu)$ and $p\in(1,\infty)$.
\subsection{Historical background}
\indent
The systematic study of singular integrals began with the Hilbert transform, which is given by
\begin{equation*}
    \calH f(x):=\frac{1}{\pi}{\rm p.v.}\int_{\RR}\frac{f(x-y)}{y}{\rm d}y,\quad x\in\RR.
\end{equation*}
The first questions concerning the operator $\calH$ were about it boundedness on $L^p(\RR)$ with $p\in(1,\infty)$. First answer to that question was given by M. Riesz in mid 1920's. His approach relied heavily on some properties of analytic functions. In early 1950's Calderón and Zygmund in their groundbreaking paper \cite{CZ} developed a real-variable method which allowed to study singular integrals in higher dimensions and resulted in introducing Calderón--Zygmund operators of the form 
\begin{equation*}
    T f(x):={\rm p.v.}\int_{\RR^d}f(x-y)K(y){\rm d}y,\quad x\in\RR^d,
\end{equation*}
where $K\colon\RR^{d}\setminus\{0\} \to \RR$ is called a Calderón--Zygmund kernel.

An important tool in studying the singular integrals is their truncated form and its convergence. Namely, for $t>0$ we define the truncated Hilbert transform
\begin{equation}
    \calH_tf(x):={\rm p.v.}\frac{1}{\pi}\int_{|y|<t}\frac{f(x-y)}{y}{\rm d}y,\quad x\in\RR
\end{equation}
as well as the truncated Calderón--Zygmund operator
\begin{equation*}
    T_t^{\rm ball}f(x):={\rm p.v.}\int_{|y|<t}f(x-y)K(y){\rm d}y,\quad x\in\RR^d.
\end{equation*}
As is well known, one of the most important steps in proving convergence of a family of operators is the boundedness on $L^p$ of maximal functions associated with it. In our case, this means showing the following inequalities
\begin{equation*}
    \norm[\big]{\sup_{t>0}|\calH_tf|}_{L^p(\RR)}\lesssim_p \|f\|_{L^p(\RR)}\quad\text{and}\quad  \norm[\big]{\sup_{t>0}|T_t^{\rm ball}f|}_{L^p(\RR^d)}\lesssim_{p,d} \|f\|_{L^p(\RR^d)}
\end{equation*}
hold for every $f\in L^p$. It is known that the above inequalities hold for $p\in(1,\infty)$.

It turns out that the question about convergence and the boundedness of the maximal function can be reduced to study only the oscillation inequality. The oscillation seminorm \eqref{eq:45} arises in the work of Bourgain \cite{B1,B2,B3} on the convergence almost everywhere of the ergodic averages along polynomials. The connection established by Bourgain between the convergence and the oscillation seminorm resulted in investigation of the oscillation inequalities (and other seminorm inequalities) for various operators, see \cite{jkrw,JR1,JRW,JRW2,MSS} and the references given there. In particular, motivated by Bourgain's work, Campbell, Jones, Reinhold and Wierdl \cite[Theorem 1.1, p. 59]{CJRW1} investigated oscillation inequalities for the truncated Hilbert transform. They proved that for any $p\in(1,\infty)$ there is a constant $C_p>0$ such that
\begin{equation}\label{hist2}
    \sup_{N\in\NN}\sup_{I\in\mathfrak{S}_N(\RR_+)}\norm[\big]{O_{I,N}^2(\calH_t f:t\in\RR_+)}_{L^p(\RR)}\le C_{p}\norm{f}_{L^p(\RR)},\quad f\in L^p(\RR).
\end{equation}
In \cite[Theorem A, p. 2116]{CJRW2} the authors managed to extend the above result to the case of multidimensional singular integrals of the Calderón--Zygmund type. Namely, for every $p\in(1,\infty)$, there is a constant $C_{p,d}>0$ such that
\begin{equation}\label{hist1}
    \sup_{N\in\NN}\sup_{I\in\mathfrak{S}_N(\RR_+)}\norm[\big]{O_{I,N}^2(T_t^{\rm ball} f:t\in\RR_+)}_{L^p(\RR^d)}\le C_{p,d}\norm{f}_{L^p(\RR^d)},\quad f\in L^p(\RR^d).
\end{equation}
In the proof of the inequality~\eqref{hist1} the authors used the symmetry of standard Euclidean ball and the fact that the operator $T_t^{\rm ball}$ is related to the group of dilations given by
\begin{equation*}
    \delta_t(x):=(tx_1,tx_2,\ldots,tx_d),\quad x\in\RR^d
\end{equation*}
to reduce matters to the inequality~\eqref{hist2}. In order to prove \eqref{hist2} they approximated $\calH_t$ by the standard dyadic martingale and used the fact that the oscillation inequality for martingales is known.

\indent The operator $\calH_t$ defined in \eqref{histRadoncon1} is a truncated form of the full continuous singular Radon transform given by
\begin{equation*}
    \calH f(x)={\rm p.v.}\int_{\RR^k}f(x-(y)^\Gamma)K(y){\rm d}y.
\end{equation*}
It is well known that the operator $\calH$ is bounded on $L^p(\RR^d)$ with $p\in(1,\infty)$. The singular Radon transforms can be seen as a natural extension of the Calderón--Zygmund theory of singular integrals. It originates in some problems related to curvatures and parabolic differential equations, see \cite{CNSW,IMSW,SW3,SW1}. Therefore, the inequality~\eqref{coscil} can be seen as a generalisation of the result of Campbell, Jones, Reinhold and Wierdl to the context of singular integrals of the Calderón--Zygmund type associated with the polynomial mappings and convex sets.

\indent The operator $\calH$, as well as it truncated form \eqref{histRadoncon1}, is related to the following group of dilations 
\begin{equation}\label{hist:dila}
    \delta_t(x):=\big(t^{|\gamma|}x_{\gamma}: \gamma\in \Gamma\big),\quad x\in\RR^\Gamma.
\end{equation}
As one can see those dilations acts different on the each coordinate. This fact prevents from reducing matters to one dimensional case. Fortunately, Jones, Seeger and Wright \cite{jsw} in their study of jumps inequalities and $r$-variations managed to overcome this obstacle. Instead of reducing to the one-dimensional case and then approximating by martingales they went straight to approximation by martingales, but this time associated to Christ's dyadic cubes (which are related to the group of dilations~\eqref{hist:dila}). We follow that argument in our proof of Theorem~\ref{Thm2}.

\indent The operator $H_t^\calP$ is a truncated form of the full discrete singular Radon transform given by
\begin{equation}\label{eq:H1}
    H^\calP f(x):=\sum_{y\in\ZZ^k\setminus\{0\}}f(x-\calP(y))K(y).
\end{equation}
The operator $H^\calP$ is a discrete counterpart of the continuous singular Radon transform $\calH^\calP$. It was a challenging problem to establish boundedness of $H^\calP$ on $\ell^p(\ZZ^d)$ with $p\in(1,\infty)$. The first partial answer was given by Stein and Wainger in \cite{SW1} where they managed to prove that $H^\calP$ is bounded on $\ell^p(\ZZ^d)$ for $p$ in a certain neighbourhood of $2$. The full range of $p\in(1,\infty)$ were obtained by Ionescu and Wainger \cite{IW}, see also \cite{M10} for a different approach.

One needs to be careful while dealing with discrete analogues of the continuous operators. Some results can be easily transfered form the continuous to the discrete setting. For example, the $\ell^p$-bound for the discrete version of the usual Calderón--Zygmund operator can be easily deduced from the bound of its continuous counterpart on $L^p$. Unfortunately, this kind of approach is impossible in the case of Radon operators, since (roughly) the polynomial $(x)^\Gamma$ may have an unbounded gaps. Also the behaviour of the discrete objects may be totally different then their continuous cousins. In order to illustrate these differences let us focus our attention on the one dimensional case and let consider $\Omega=(-1,1)$, $(x)^\Gamma=x^3$ and $K(y)=y^{-1}$. Then the Fourier multiplier associated with the operator $H_t$ is of the form
\begin{equation*}
    m_t(\xi)=\sum_{y\in(-t,t)\cap\ZZ\setminus\{0\}}e(\xi y^3)y^{-1},
\end{equation*}
while in the case of $\calH_t$ the multiplier is given by
\begin{equation*}
    \psi_t(\xi)={\rm p.v.}\int_{-t}^te(\xi y^3)y^{-1}{\rm d}y.
\end{equation*}
It can be easily shown that
\begin{equation*}
    \psi_t(\xi)=i\frac{2}{3}{\rm sgn}(\xi){\rm Si}(2\pi \xi t^3), 
\end{equation*}
where ${\rm Si}(z)$ is the Sine Integral. Meanwhile in the discrete case we have $m_t(\xi)=0$ when $\xi\in\ZZ$. As we can notice the sum and the integral demonstrate completely different behaviours and consequently we need to use different approach to study discrete analogues of Radon operators. 

\indent The oscillation inequality \eqref{eq:3} for the discrete Radon transform seems to be new, even in the case $k=d=1$. In the proof of Theorem~\ref{Thm1} we exploit the technique based on subexponential sequences and the Rademacher--Menshov inequality which was introduced in \cite{MSZ3}, see also \cite{MST2}. Moreover, we make use of the vector-valued Ionescu--Wainger multiplier theorem and the square function estimates from \cite{MSZ2}.

\indent At the end, it is worth to mention that similar seminorm inequalities, like variational estimates or the jump inequality, for both $H_t^\calP$ and $\calH^\calP_t$, are already known, see \cite{jsw,MST2,MSZ2,MSZ3}. However, those known estimates do not imply our results.
\section{Preliminaries}\label{sec:notation}
\subsection{Basic notation}
We denote $\NN:=\{1, 2, \ldots\}$ and $\NN_0:=\{0,1,2,\ldots\}$. For
$d\in\NN$ the sets $\ZZ^d$, $\RR^d$, $\CC^d$ and
$\TT^d:=\RR^d/\ZZ^d$ have the standard meaning. We denote $\RR_+:=(0, \infty)$ and
for every $N\in\NN$ we  define
    \[
    \NN_N:=\{1,\ldots, N\}\text{ and } \NN_\infty:=\NN
    \]
For $\tau\in(0, 1)$ and $u\in\NN$ we define sets
    \begin{align*}
        \DD_{\tau}:=\{n^\tau:n\in\NN\}\quad\text{ and }\quad 2^{u\NN}:=\{2^{un}\colon n\in\NN\}.
    \end{align*}
We write $A \lesssim B$ to indicate that $A\le CB$ with a constant $C>0$. The constant $C$ may vary from line to line. We write $\lesssim_{\delta}$ if the implicit constant depends on $\delta$. 

For $N\in\NN\cup\{\infty\}$ we denote by $\mathfrak S_N(\II)$ the family of all strictly increasing
sequences of length $N+1$ contained in $\II$.

\subsection{Norms and Fourier transform} The standard inner product and the corresponding Euclidean norm on $\RR^d$ are denoted by $ x\cdot\xi$ and $\abs{x}:=\abs{x}_2$, respectively. Furthermore, we need the maximum norm on $\RR^d$ defined as $$|x|_{\infty}:=\max_{1\leq k\leq d}|x_k|.$$

For any $\gamma=(\gamma_1,\dots,\gamma_k)\in\N^k$ we abuse the notation by writing $|\gamma|:=\gamma_1+\cdots+\gamma_k$. However, this should never cause any confusions
since the multi-indices will be always denoted by Greek letter $\gamma$.

Let $\GG=\RR^d$ or $\GG=\ZZ^d$. By the symbol $\GG^*$ we denote the dual group of $\GG$. The Fourier transform $\calF_{\GG}$ of any $f \in L^1(\GG)$ is defined by
\begin{align*}
\calF_{\GG} f(\xi) := \int_{\GG} f(x) \ex(x\cdot\xi) {\rm d}\mu(x),\quad \xi\in\GG^*,
\end{align*}
where $\mu$ is the Haar measure on $\GG$ and $\ex(\xi):=e^{2\pi {\bm i} \xi}$ with ${\bm i}^2=-1$. The Fourier multiplier operator associated with the bounded function $\mathfrak m\colon\GG^*\to\CC$ is given by
\begin{align}
\label{eq:2}
T_{\GG}[\mathfrak m]f(x):=\int_{\GG^*}\ex(-\xi\cdot x)\mathfrak m(\xi)\calF_{\GG}f(\xi){\rm d}\xi, \quad x\in\GG.
\end{align}
Here, we assume that $f\colon\GG\to\CC$ is a compactly supported function on $\GG$ (and smooth if $\GG=\RR^d$) or any other function for which \eqref{eq:2} makes sense.
\subsection{Ionescu--Wainger multiplier theorem}\label{sec:IW}
In order to handle the discrete singular integral we make use of the Ionescu--Wainger multiplier theorem.
\begin{theorem}\label{thm:IW-mult}
For every $\varrho>0$, there exists a family $(P_{\leq N})_{N\in\NN}$ of subsets of $\NN$ such that:  
\begin{enumerate}[label*={(\roman*)}]
\item \label{IW1} One has $\NN_N\subseteq P_{\leq N}\subseteq\N_{\max\{N, e^{N^{\varrho}}\}}$.
\item \label{IW2}  If $N_1\le N_2$, then $P_{\leq N_1}\subseteq P_{\leq N_2}$.
\item \label{IW3}  If $q \in P_{\leq N}$, then all factors of $q$ also lie in $P_{\leq N}$.
\end{enumerate}

Furthermore, for every $p \in (1,\infty)$, there exists  $0<C_{p, \varrho, d}<\infty$ such that, for every $N\in\NN$, the following holds.

Let $0<\varepsilon_N \le e^{-N^{2\varrho}}$, and let $\Theta\colon\RR^{d} \to L(H_0,H_1)$ be a measurable function supported on $\varepsilon_{N}\mathbf Q$, where $\mathbf Q:=[-1/2, 1/2)^d$ is a unit cube, with values in the space $L(H_{0},H_{1})$ of bounded linear operators between separable Hilbert spaces $H_{0}$ and $H_{1}$.
Let $0 \leq \mathbf A_{p} \leq \infty$ denote the smallest constant such that, for every function $f\in L^2(\RR^d;H_0)\cap L^{p}(\RR^d;H_0)$, we have
\begin{align}
\label{eq:75}
\norm[\big]{T_{\RR^d}[\Theta]f}_{L^{p}(\RR^{d};H_1)}
\leq
\mathbf A_{p} \norm{f}_{L^{p}(\RR^{d};H_0)}.
\end{align}
Then the multiplier
\begin{equation}
\label{eq:IW-mult}
\Delta_N(\xi)
:=\sum_{b \in\Sigma_{\leq N}}
\Theta(\xi - b),
\end{equation}
where $\Sigma_{\leq N}$ is defined by
\begin{align}
\label{eq:42}
\Sigma_{\leq N} := \Big\{ \frac{a}{q}\in\QQ^d\cap\TT^d:  q \in P_{\leq N}\text{ and } {\rm gcd}(a, q)=1\Big\},
\end{align}
satisfies for every $f\in L^p(\ZZ^d;H_0)$ the following inequality
\begin{align}
\label{eq:76}
\norm[\big]{ T_{\ZZ^d}[\Delta_{N}]f}_{\ell^p(\ZZ^{d};H_1)}
\le C_{p,\varrho,d}
(\log N) \mathbf A_{p}
\norm{f}_{\ell^p(\ZZ^{d};H_0)}.
\end{align}
\end{theorem}
The reader can found the proof of Theorem~\ref{thm:IW-mult} in \cite[Section 2]{MSZ3}. We also refer to
the paper of Tao \cite{TaoIW}, where he was able to remove the factor $\log N$ from \eqref{eq:76}.
\section{Proof of Theorem~\ref{Thm1} for the canonical mapping $(x)^\Gamma$}
Assume that $p\in(1,\infty)$ and let $f\in \ell^p(\ZZ^\Gamma)$ be a function with a compact support. By using the monotone convergence theorem and standard density arguments to prove \eqref{oscileqcanon} it is enough to establish
\begin{equation}\label{eq:40} 
    \sup_{N\in\NN}\sup_{I\in\mathfrak{S}_N(\II)}\norm{O_{I,N}^2(H_tf: t\in\II)}_{\ell^p(\ZZ^\Gamma)}\le C_{p,k,|\Gamma|}\norm{f}_{\ell^p(\ZZ^\Gamma)}
\end{equation}
for every finite subset $\II\subset\RR_+$ with a constant $C_{p,k,|\Gamma|}> 0$ that is independent of the set $\II$. Let us choose $p_0\in(1,2)$, close to $1$ such that $p\in(p_0,p_0')$. Then we take $\tau\in(0,1)$ such that
\begin{equation}\label{eq:5}
    \tau<\frac{1}{2}\min\{p_0-1,1\}.
\end{equation}
We split \eqref{eq:40} (cf. \cite[Lemma 1.3]{jsw}) into long oscillations and short variations
\begin{equation}\label{eq:4}
\begin{aligned}
   \sup_{N\in\NN}\sup_{I\in\mathfrak{S}_N(\II)}\norm{O_{I,N}^2(H_tf: t\in\II)}_{\ell^p(\ZZ^\Gamma)}\lesssim&\sup_{N\in\NN}\sup_{I\in\mathfrak{S}_N(\DD_\tau)}\norm[\big]{O_{I,N}^2(H_{2^{n^\tau}}f:n\in\NN_0)}_{\ell^p(\ZZ^\Gamma)}\\
    &+\norm[\Big]{\Big(\sum_{n=0}^\infty V^2\big(H_{t}f:t\in[2^{n^\tau},2^{(n+1)^\tau})\cap\II\big)^2\Big)^{1/2}}_{\ell^p(\ZZ^\Gamma)}.
\end{aligned}
\end{equation}
The short variations were estimated in \cite[Section 3.1]{MSZ3}. Hence we are reduced to prove only the estimate for long oscillations.
\subsection{Estimates for long oscillations}
Our aim is to establish the following inequality
\begin{equation}\label{eq:49}
    \sup_{N\in\NN}\sup_{I\in\mathfrak{S}_N(\DD_\tau)}\norm[\big]{O_{I,N}^2(H_{2^{n^\tau}}f:n\in\NN_0)}_{\ell^p(\ZZ^\Gamma)}\lesssim\|f\|_{\ell^p(\ZZ^\Gamma)}.
\end{equation}
Let us observe that for any $x\in\ZZ^\Gamma$ one has $H_{2^{n^\tau}}f(x)=T_{\ZZ^\Gamma}[m_{{n^\tau}}]f(x),$ where
\begin{equation*}
    m_{n^\tau}(\xi)=\sum_{y \in\Omega_{2^{n^\tau}}\cap\ZZ^{k}\setminus\{0\}}e(\xi\cdot (y)^\Gamma)K(y),\quad\xi\in\TT^\Gamma.
\end{equation*}
\indent Now, let $\chi\in(0,1/10)$. The proof of \eqref{eq:49} will require several appropriately chosen parameters. Choose $\alpha>0$ such that
\begin{equation*}
    \alpha>10\left(\frac{1}{p_0}-\frac{1}{2}\right)\left(\frac{1}{p_0}-\frac{1}{\min\{p,p'\}}\right)^{-1}.
\end{equation*}
Let $u\in\NN$ be a large natural number which will be specified later. We set
\begin{equation}
    \varrho:=\min\left\{\frac{1}{10u},\frac{\delta}{8\alpha}\right\},
\end{equation}
where $\delta>0$ is from the estimate for the Gauss sum \eqref{gaussum}. Now, let us consider $S_0=\max\{t>0:t\in 2^{u\NN}\cap[1, n^{\tau u}]\}$. We recall the family of rational fractions $\Sigma_{\leq S_0}$ related to parameter $\varrho$, defined in Section~\ref{sec:IW}. For simplicity we will write
\begin{equation*}
    \Sigma_{\le n^{\tau u}}:=\Sigma_{\leq S_0}.
\end{equation*}
Next, for dyadic integers $S\in2^{u\NN}$ we define
\begin{equation*}
    \Sigma_S=\begin{cases}
    \Sigma_{\leq S},& {\rm if}\quad S=2^u,\\
    \Sigma_{\leq S}\setminus\Sigma_{\leq S/2^u},& {\rm if}\quad S>2^u.
    \end{cases}
\end{equation*}
It is easy to see that
\begin{equation}\label{eq:52}
     \Sigma_{\le n^{\tau u}}=\bigcup_{\substack{S\le n^{\tau u},\\S\in 2^{u\NN}}}\Sigma_S.
\end{equation}
Now, we define the Ionescu--Wainger projection multipliers. For this purpose, we introduce a diagonal matrix $A$ of size $|\Gamma| \times |\Gamma|$ given by the condition $(A v)_\gamma := \abs{\gamma} v_\gamma$ for any $\gamma \in \Gamma$ and $v\in\RR^\Gamma$. Moreover, for any $t > 0$ we also define
corresponding dilations by setting $t^A x=\big(t^{|\gamma|}x_{\gamma}: \gamma\in \Gamma\big)$ for
every $x\in\RR^\Gamma$. Let $\eta\colon\RR^\Gamma\to[0,1]$ be a smooth function such that
\begin{equation*}
    \eta(x)=\begin{cases}
    1, & |x|\le1/(32|\Gamma|), \\
    0, & |x|\ge1/(16|\Gamma|).\end{cases}
\end{equation*}
For any $n\in\NN$ we set
\begin{equation}\label{eq:50}
    \Pi_{\le n^\tau}(\xi):=\sum_{a/q\in\Sigma_{\le n^{\tau u}}}\eta\big(2^{n^\tau(A-\chi I)}(\xi-a/q)\big),\quad \xi\in\TT^\Gamma.
\end{equation}
It is easy to check that Theorem~\ref{thm:IW-mult} holds for the multiplier $\Pi_{\le n^\tau}$ since one has $2^{-n^\tau(|\gamma|-\chi)}\le e^{-n^{\tau/10}}\le e^{-S_0^\varrho}$ for sufficiently large $n\in\NN$.\\
\indent Using the functions defined in \eqref{eq:50} we can partition the multiplier $m_{{n^\tau}}$ and estimate the left hand side of \eqref{eq:49} by
\begin{align}
    &\sup_{N\in\NN}\sup_{I\in\mathfrak{S}_N(\DD_\tau)}\norm[\Big]{O_{I,N}^2\Big(T_{\ZZ^\Gamma}\big[\sum_{j=1}^n(m_{{j^\tau}}-m_{(j-1)^\tau})\Pi_{\le j^\tau}\big]f:n\in\NN_0\Big)}_{\ell^p(\ZZ^\Gamma)}\label{eq:12}\\
    &+\sup_{N\in\NN}\sup_{I\in\mathfrak{S}_N(\DD_\tau)}\norm[\Big]{O_{I,N}^2\Big(T_{\ZZ^\Gamma}\big[\sum_{j=1}^n(m_{{j^\tau}}-m_{(j-1)^\tau})(1-\Pi_{\le j^\tau})\big]f:n\in\NN_0\Big)}_{\ell^p(\ZZ^\Gamma)}\label{eq:11}.
\end{align}
Here we use convection that for $n=0$ the sum is equal to $0$. The expressions in \eqref{eq:12} and \eqref{eq:11} correspond to major and minor arcs from the Hardy--Littlewood circle method, respectively.
\subsubsection{Minor arcs}
First, we note that the oscillation seminorm \eqref{eq:11} is controlled by 1-variation $V^1$ and by the following estimate
\begin{align*}
    \norm[\Big]{V^1\big(T_{\ZZ^\Gamma}\big[\sum_{j=1}^n(m_{{j^\tau}}-m_{(j-1)^\tau})&(1-\Pi_{\le j^\tau})\big]f:n\in\NN_0\big)}_{\ell^p(\ZZ^\Gamma)}\\
    &\lesssim\sum_{n=0}^\infty\norm{T_{\ZZ^\Gamma}[(m_{{(n+1)^\tau}}-m_{n^\tau})(1-\Pi_{\le (n+1)^\tau})]f}_{\ell^p(\ZZ^\Gamma)}.
\end{align*}
Consequently, it is enough to show
\begin{equation}\label{eq:9}
    \|T_{\ZZ^\Gamma}[(m_{{(n+1)^\tau}}-m_{n^\tau})(1-\Pi_{\le (n+1)^\tau})]f\|_{\ell^p(\ZZ^\Gamma)}\lesssim(n+1)^{-2}\|f\|_{\ell^p(\ZZ^\Gamma)}.
\end{equation}
To prove the above estimate we use Weyl's inequality \cite[Theorem A.1, p. 49]{MSZ3} and proceed as in \cite[Lemma 3.29, p. 34]{MSZ3} (provided that $u\in\NN$ is large enough).
\subsubsection{Major arcs and multiplier approximation}
Now we need to prove that the estimate
\begin{align}\label{eq:44}
   \sup_{N\in\NN}\sup_{I\in\mathfrak{S}_N(\DD_\tau)}\norm[\Big]{O_{I,N}^2\Big(T_{\ZZ^\Gamma}\big[\sum_{j=1}^n(m_{{j^\tau}}-m_{(j-1)^\tau})\Pi_{\le j^\tau}\big]f:n\in\NN_0\Big)}_{\ell^p(\ZZ^\Gamma)}\lesssim\norm{f}_{\ell^p(\ZZ^\Gamma)}
\end{align}
holds. At first we show that the multiplier
\begin{equation*}
    \sum_{a/q\in\Sigma_{\le j^{\tau u}}}\big(m_{j^\tau}-m_{(j-1)^\tau}\big)(\xi)\eta(2^{j^\tau(A-\chi I)}(\xi-a/q)).
\end{equation*}
is, up to an acceptable error term, equal to
\begin{equation}\label{eq:38}
    {\bm m}_j(\xi)=\sum_{a/q\in\Sigma_{\le j^{\tau u}}}G(a/q)\big(\Psi_{2^{j^\tau}}-\Psi_{2^{(j-1)^\tau}}\big)(\xi-a/q)\eta(2^{j^\tau(A-\chi I)}(\xi-a/q))
\end{equation}
where $\Psi_{t}$ is a continuous version of the multiplier $m_{t}$ given by
\begin{align*}
    \Psi_t(\xi)&:={\rm p.v.}\int_{\Omega_{t}}{e(\xi\cdot(y)^\Gamma)K(y){\rm d}y}
\end{align*}
and $G(a/q)$ is the Gauss sum defined by
\begin{equation*}
    G(a/q):=\frac{1}{q^k}\sum_{r\in\N_q^k}{e((a/q)\cdot(r)^\Gamma)}.
\end{equation*}
\indent Let us state some properties of the function $\Psi_t$ and the Gauss sum $G(a/q)$ which will be used later on. For a fixed $c\in(0,1)$ and any real number $t>0$ we have the following estimates for the function $\Psi_t$:
\begin{equation}\label{Phi}
    |\Psi_t(\xi)-\Psi_{ct}(\xi)|\lesssim |t^{A}\xi|_\infty\quad\text{and}\quad|\Psi_t(\xi)-\Psi_{ct}(\xi)|\lesssim |t^{A}\xi|_\infty^{-\sigma/|\Gamma|},
\end{equation}
where the first estimate follows from the cancellation condition \eqref{eq:cancel} and the second one is a consequence of van der Corput's lemma (see \cite[p. 21]{MSZ2} for details). It is well known (see \cite[Lemma 4.14, p. 44]{MSZ3}) that for some $\delta>0$ one has
\begin{equation}\label{gaussum}
    |G(a/q)|\lesssim_k q^{-\delta}.
\end{equation}
\indent Now, we claim that one has
\begin{align}\label{l2approx:3}
    \norm[\Big]{T_{\ZZ^\Gamma}\big[\sum_{j=1}^n (m_{j^\tau}-m_{(j-1)^\tau})\Pi_{\leq j^\tau}-{\bm m}_j\big]f}_{\ell^p(\ZZ^\Gamma)}\lesssim 2^{-j^\tau\sigma/4}\norm{f}_{\ell^p(\ZZ^\Gamma)}.
\end{align}
To see that \eqref{l2approx:3} holds we follow the approach taken in \cite[Lemma 3.38, p. 36]{MSZ3} by appealing to conditions \eqref{eq:size-unif} and \eqref{eq:K-modulus-cont} and Theorem~\ref{thm:IW-mult}. As a result, to prove \eqref{eq:44} it suffices to establish that
\begin{equation}\label{eq:35}
     \sup_{N\in\NN}\sup_{I\in\mathfrak S_N(\DD_\tau)}\norm[\big]{O_{I,N}^2\big(T_{\ZZ^\Gamma}\big[\sum_{j=1}^n{\bm m}_{j}\big]f:n\in\NN_0\big)}_{\ell^p(\ZZ^\Gamma)}\lesssim\norm{f}_{\ell^p(\ZZ^\Gamma)}.
\end{equation}
\indent In the next step, we use \eqref{eq:52} to write that
\begin{equation}\label{eq:36}
    {\bm m}_j(\xi)=\sum_{\substack{S\le j^{\tau u},\\S\in 2^{u\NN}}} {\bm m}_S^j(\xi),
\end{equation}
where ${\bm m}_S^j$ is defined as
\begin{equation}\label{eq:20}
    {\bm m}_S^j(\xi):=\sum_{a/q\in\Sigma_S}G(a/q)\big(\Psi_{2^{j^\tau}}-\Psi_{2^{(j-1)^\tau}}\big)(\xi-a/q)\eta\big(2^{j^\tau(A-\chi I)}(\xi-a/q)\big).
\end{equation}
Consequently, we see that to prove \eqref{eq:35} it is enough to show that
\begin{equation}\label{eq:21}
    \sup_{N\in\NN}\sup_{I\in\mathfrak{S}_N(\DD_\tau^S)}\norm[\big]{O_{I,N}^2(T_{\ZZ^\Gamma}\big[\sum_{\substack{1\le j\le n\\ S^{1/u}\le j^\tau}}{\bm m}_S^j\big]f:n^\tau\ge S^{1/u})}_{\ell^p(\ZZ^\Gamma)}\lesssim S^{-3\varrho}\norm{f}_{\ell^p(\ZZ^\Gamma)},
\end{equation}
where $\DD_\tau^S=\{n\in\NN:n^\tau\ge S^{1/u}\}$ since $S^{-3\varrho}$ is summable in $S\in 2^{u\NN}$.
\subsubsection{Gaussian multiplier and scale distinction}
In order to apply Theorem~\ref{thm:IW-mult} we need to get rid of the Gaussian part $G(a/q)$ in the multiplier \eqref{eq:20}. For this purpose let us define $\Tilde{\eta}(x):=\eta(x/2)$ and two new multipliers
\begin{align*}
    v_S^j(\xi)&:=\sum_{a/q\in\Sigma_S}(\Psi_{2^{j^\tau}}-\Psi_{2^{(j-1)^\tau}})(\xi-a/q)\eta\big(2^{j^\tau(A-\chi I)}(\xi-a/q)\big),\\
    \mu_S(\xi)&:=\sum_{a/q\in\Sigma_S}G(a/q)\Tilde{\eta}\big(2^{S^{1/u}(A-\chi I)}(\xi-a/q)\big).
\end{align*}
We note that one has ${\bm m}_S^j=v_S^j\mu_S$ and consequently to prove \eqref{eq:21} it suffices to show that for every $p\in(1,\infty)$ one has
\begin{align}
    \norm[\big]{T_{\ZZ^\Gamma}[\mu_S]f}_{\ell^p(\ZZ^\Gamma)}&\lesssim S^{-6\varrho}\norm{f}_{\ell^p(\ZZ^\Gamma)},\label{eq:22}\\
    \sup_{N\in\NN}\sup_{I\in\mathfrak S_N(\DD_\tau^S)}\norm[\big]{O_{I,N}^2\big(T_{\ZZ^\Gamma}\big[\sum_{\substack{1\le j\le n\\ S^{1/u}\le j^\tau}}v_S^j\big]f:n^\tau\geq S^{1/u}\big)}_{\ell^p(\ZZ^\Gamma)}&\lesssim S^{3\varrho}\norm{f}_{\ell^p(\ZZ^\Gamma)}\label{eq:23}.
\end{align}
The Gauss sum estimates \eqref{eq:22} were established in \cite[Lemma 3.47 and estimate (3.49) pp. 38--39]{MSZ3}. Hence we may focus on \eqref{eq:23}. Let $\kappa_S=\ceil{S^{2\varrho}}$. We split the left hand side of (\ref{eq:23}) at point $2^{\kappa_S}$ and write
\begin{align*}
    {\rm LHS}(\ref{eq:23})\lesssim&\sup_{N\in\NN}\sup_{I\in\mathfrak S_N(\DD^\tau_{\le S})}\norm[\big]{O_{I,N}^2\big(T_{\ZZ^\Gamma}\big[\sum_{\substack{1\le j\le n\\ S^{1/(\tau u)}\le j}}v_S^j\big]f:n\in[S^{1/(\tau u)},2^{\kappa_S+1}]\big)}_{\ell^p(\ZZ^\Gamma)}\\
    &+\sup_{N\in\NN}\sup_{I\in\mathfrak S_N(\DD_{\ge S})}\norm[\big]{O_{I,N}^2\big(T_{\ZZ^\Gamma}\big[\sum_{\substack{1\le j\le n\\2^{\kappa_S}\le j}}v_S^j\big]f:n\ge2^{\kappa_S})}_{\ell^p(\ZZ^\Gamma)},
\end{align*}
where $\DD^\tau_{\le S}:=\{n\in\NN:n\in[S^{1/(\tau u)},2^{\kappa_S+1}]\}$ and  $\DD_{\ge S}:=\{n\in\NN:n\ge2^{\kappa_S}\}$. We handle with small scales by using the Rademacher--Menschov inequality and Theorem~\ref{thm:IW-mult}. In the case of large scales we make use of the transference principle due to Magyar--Stein--Wainger from \cite{MSW} to obtain estimates from the continuous case. 
\subsubsection{Estimates for small scales}
We closely follow approach taken in \cite{MSZ3} to prove the following estimate
\begin{equation}\label{eq:14}
    \sup_{N\in\NN}\sup_{I\in\mathfrak S_N(\DD^\tau_{\le S})}\norm[\big]{O_{I,N}^2\big(T_{\ZZ^\Gamma}\big[\sum_{\substack{1\le j\le n\\ S^{1/(\tau u)}\le j}}v_S^j\big]f:n\in[S^{1/(\tau u)},2^{\kappa_S+1}]\big)}_{\ell^p(\ZZ^\Gamma)}\lesssim \kappa_S\log(S)\norm{f}_{\ell^p(\ZZ^\Gamma)}.
\end{equation}
By applying the Rademacher--Menshov inequality \cite[Lemma 2.5, p. 534]{MSZ2} to the left hand side of \eqref{eq:14} we see that
\begin{equation*}
    {\rm LHS}(\ref{eq:14})\lesssim\sum_{i=0}^{\kappa_S+1}\norm[\bigg]{\Big(\sum_{j}\big|\sum_{k\in I_j^i}T_{\ZZ^\Gamma}[v_S^k]f\big|^2\Big)^{1/2}}_{\ell^p(\ZZ^\Gamma)},
\end{equation*}
where $I_j^i:=[j2^i,(j+1)2^i)\cap [S^{1/(\tau u)},2^{\kappa_S+1}]$. Hence it suffices to prove that for every $i\le \kappa_S+1$ we have
\begin{align}
    \norm[\bigg]{\Big(\sum_{j}\big|\sum_{k\in I_j^i}T_{\ZZ^\Gamma}[v_S^k]f\big|^2\Big)^{1/2}}_{\ell^p(\ZZ^\Gamma)}\lesssim\log(S)\norm{f}_{\ell^p(\ZZ^\Gamma)}.\label{eq:15}
\end{align}
By Theorem~\ref{thm:IW-mult}, the estimate \eqref{eq:15} is a consequence of its continuous counterpart
\begin{align*}
\norm[\bigg]{ \Big(\sum_{j}\big|\sum_{k\in I_j^i}T_{\RR^\Gamma}\big[(\Psi_{2^{k^\tau}}-\Psi_{2^{(k-1)^\tau}})\eta\big(2^{k^\tau(A-\chi I)}\cdot\big)\big]f\big|^2\Big)^{1/2}}_{L^p(\RR^\Gamma)}
\lesssim \norm{f}_{L^p(\RR^\Gamma)}.
\end{align*}
The above square function estimate follows by  appealing to \eqref{Phi} and standard arguments from the Littlewood--Paley theory. We refer to \cite{MSZ2} for more details, see also discussion below \cite[Theorem 4.3, p. 42]{MSZ3}.  
\subsubsection{Estimates for large scales}
The last thing to show is the estimate for the large scales,
\begin{equation}\label{eq:30}
    \sup_{N\in\NN}\sup_{I\in\mathfrak S_N(\DD_{\ge S})}\norm[\big]{O_{I,N}^2\big(T_{\ZZ^\Gamma}\big[\sum_{\substack{1\le j\le n\\2^{\kappa_S}\le j}}v_S^j\big]f:n\ge2^{\kappa_S})}_{\ell^p(\ZZ^\Gamma)}\lesssim\log(S)\norm{f}_{\ell^p(\ZZ^\Gamma)}.
\end{equation}
We would like to exploit the almost telescoping nature of the multipliers appearing in \eqref{eq:30}. We do this by introducing new approximating multipliers. Let
\begin{equation*}
    \Tilde{v}_S^j(\xi):=\sum_{a/q\in\Sigma_S}(\Psi_{2^{j^\tau}}-\Psi_{2^{(j-1)^\tau}})(\xi-a/q)\eta(2^{2^{\tau\kappa_S}(A-\chi)}(\xi-a/q)).
\end{equation*}
Since $j\ge 2^{\kappa_S}$, the expression
\begin{equation*}
    \eta(2^{j^\tau(A-\chi I)}(\xi-a/q))-\eta(2^{2^{\tau\kappa_S}(A-\chi I)}(\xi-a/q))
\end{equation*}
is nonzero only when $|\xi_\gamma-a_\gamma/q|\gtrsim2^{-j^\tau(|\gamma|-\chi)}$ for some $\gamma\in\Gamma$. Hence, by the van der Corput estimate in (\ref{Phi}) we get
\begin{equation*}
    \norm[\big]{T_{\ZZ^\Gamma}[v_S^j-\Tilde{v}_S^j]f}_{\ell^2(\ZZ^\Gamma)}\lesssim 2^{-j^\tau\chi\sigma/|\Gamma|}\|f\|_{\ell^2(\ZZ^\Gamma)},
\end{equation*}
whereas for any  $p\neq 2$, by property \ref{IW1}, we have 
\begin{equation*}
    \norm[\big]{T_{\ZZ^\Gamma}[v_S^j-\Tilde{v}_S^j]f}_{\ell^p(\ZZ^\Gamma)}\lesssim\big|\Sigma_{\leq j^{\tau u}}\big|\norm{f}_{\ell^p(\ZZ^\Gamma)}\lesssim e^{(|\Gamma|+1)j^{\tau/10}}\norm{f}_{\ell^p(\ZZ^\Gamma)}.
\end{equation*}
By interpolating the above inequalities we get
\begin{equation}\label{eq:S5}
    \norm[\big]{T_{\ZZ^\Gamma}[v_S^j-\Tilde{v}_S^j]f)}_{\ell^p(\ZZ^\Gamma)}\lesssim 2^{-j^\tau\varepsilon}\|f\|_{\ell^p(\ZZ^\Gamma)}
\end{equation}
with some $\varepsilon>0$.
Consequently, the inequality \eqref{eq:30} will follow if we show
\begin{equation}\label{eq:S30}
    \sup_{N\in\NN}\sup_{I\in\mathfrak S_N(\DD^\tau_{\ge S})}\norm[\big]{O_{I,N}^2\big(T_{\ZZ^\Gamma}[\Delta_S^n]f):n\ge2^{\kappa_S})}_{\ell^p(\ZZ^\Gamma)}\lesssim\log(S)\norm{f}_{\ell^p(\ZZ^\Gamma)},
\end{equation}
where
\begin{equation*}
    \Delta_S^n(\xi)=\sum_{a/q\in\Sigma_S}\Psi_{2^{n^\tau}}(\xi-a/q)\eta(2^{2^{\tau\kappa_S}(A-\chi I)}(\xi-a/q)).
\end{equation*}
By using the standard transference argument (see \cite[Section 3.6, pp. 40--41]{MSZ3}) the estimate \eqref{eq:S30} is reduced to showing for every $p\in(1, \infty)$ and $f\in L^p(\RR^{\Gamma})$ the following  uniform oscillation inequality
\begin{equation}\label{eq:46}
    \sup_{N\in\NN}\sup_{I\in\mathfrak S_N(\RR_+)}\norm[\big]{O_{I,N}^2(T_{\RR^\Gamma}[\Psi_t]f:t\in\RR_+)}_{L^p(\RR^\Gamma)}\lesssim\norm{f}_{L^p(\RR^\Gamma)}
\end{equation}
which follows by Theorem~\ref{Thm2}. This completes the proof of Theorem~\ref{Thm1}.
\section{Proof of theorem~\ref{Thm2} for the canonical mapping $(x)^\Gamma$}
The proof of the inequality
\begin{equation}\label{eq:ap3}
    \sup_{N\in\NN}\sup_{I\in\mathfrak{S}_N(\RR_+)}\norm[\big]{O_{I,N}^2(\calH_t f:t\in\RR_+)}_{L^p(\RR^\Gamma)}\lesssim\|f\|_{L^p(\RR^\Gamma)}
\end{equation}
follows the same lines as the proofs of \cite[Theorem 1.1, Theorem 1.2]{jsw} and \cite[Theorem A.2]{MST2}. We will split \eqref{eq:ap3} into long oscillations and short variations. The estimate for long oscillations is obtained by approximation with a suitable dyadic martingale, whereas short variations are estimated by using the Littlewood--Paley theory. In the case of long oscillations the key ingredient is the oscillation inequality for Christ's dyadic martingales which follows from the result of Jones, Kaufman, Rosenblatt and Wierdl \cite[Theorem 6.4, p. 930]{jkrw} (see also \cite[Proposition 2.8]{MSS})
\subsection{Dyadic martingales on the homogeneous spaces}
In the context of homogeneous spaces and dyadic martingales we will follow the notation introduced in \cite{jsw}. Let $A$ be a $d\times d$ matrix whose eigenvalues have positive real parts. For any $t>0$ we consider the dilation given by
\begin{equation}\label{eq:A4}
    t^A:=\exp(A\log t).
\end{equation}
 We say that a quasi-norm $\rho\colon\RR^d\rightarrow [0, \infty)$ is homogeneous with respect to the group of dilations $(t^A:t>0)$ if $\rho(t^Ax)=t\rho(x)$ for any $x\in\RR^d$ and $t>0$. Recall, that for a given group of of dilations $(t^A:t>0)$ by \cite[Proposition 1.7, Definition 1.8]{SW3} there exists a quasi-norm $\rho$ which is homogeneous with respect to that group. We note that $\RR^d$ equipped with a homogeneous quasi-norm $\rho$ and Lebesgue measure is a space of homogeneous type with the quasi-metric induced by $\rho$.
 As it was shown by Christ {\cite[Theorem 11]{C}} for any given space of homogeneous type there exists a system of $D$-dyadic cubes $\{Q_{\alpha}^k: k\in\ZZ, \alpha\in I_k\}$, with $D>1$. For a locally integrable function $f$ the martingale sequence associated with the system of dyadic cubes $\{Q_\alpha^k\}$ is of the form
\begin{equation}\label{christmartin}
\EE_k f(x) := \EE[f | \calF_k] (x):=\frac{1}{|Q_{\alpha}^k|} \int_{Q_{\alpha}^k}f(y){\: \rm d}y
\end{equation}
provided $Q_{\alpha}^k$ is the unique dyadic cube containing $x\in\RR^d$.
\begin{theorem}[{\cite[Theorem 6.4, p. 930]{jkrw}}]\label{thm:martinoscilation}
For every $p\in(1, \infty)$ there exists a constant $C_p>0$ such that
\begin{equation*}
    \sup_{N\in\NN}\sup_{I\in\mathfrak S_N(\ZZ)}\norm[\big]{O^2_{I,N}(\EE_nf:n\in\ZZ)}_{L^p(\RR^d)}\leq C_p\norm{f}_{L^p(\RR^d)}.
\end{equation*}
\end{theorem}
The next result which follows from \cite{jsw} concern the approximation by martingales associated with Christ's dyadic cubes.
\begin{lemma}[{cf. \cite[Theorem 1.1]{jsw}}]\label{martinapprox}
    Let $\phi$ be a Schwartz function such that $\int\phi=1$ and let $\phi_{D^k}(x)=D^{-{\rm tr} (A)k}\phi(D^{-kA}x)$ where $D>1$ is associated with the system of $D$-dyadic cubes $\{Q_{\alpha}^k\}$. Then the operator
    \begin{equation*}
    \calS f(x)=\big(\sum_{k\in\ZZ}|\phi_{D^k}\ast f(x) -\mathbb{E}_k f(x)|^2\big)^{1/2}
    \end{equation*}
    is bounded on $L^p(\RR^d)$ for $p\in(1,\infty)$. Moreover, for $p=1$ the operator $\calS$ is of weak type (1,1).
\end{lemma}
\begin{proof}
The proof is a repetition of arguments presented in proofs of \cite[Lemma 3.2 and Theorem 1.1 pp. 19--20]{jsw}.
\end{proof}
The next result is a counterpart of \cite[Theorem 1.1]{jsw} in the context of the oscillation seminorm. Let $\sigma$ be a compactly supported finite Borel measure on $\RR^d$. Let us consider dilates of $\sigma$ defined by
\begin{equation}\label{eq:ap2}
    \langle \sigma_t,f\rangle=\int_{\RR^d}f(t^Ax){\rm d}\sigma(x)
\end{equation}
where $t^A$ is as in \eqref{eq:A4}. We additionally assume that the Fourier transform satisfies the following size condition
\begin{equation}\label{eq:ap1}
    |\hat{\sigma}(\xi)|\lesssim |\xi|^{-a}\qquad\text{for some}\qquad a>0.
\end{equation}
\begin{theorem}\label{oscillationmeasure}
Assume that $\sigma$ is a compactly supported finite Borel measure on $\RR^d$ satisfying \eqref{eq:ap1} for some $a>0$. Let $\sigma_t$ be as in \eqref{eq:ap2}. Then for $p\in(1,\infty)$ one has
\begin{equation}\label{eq:A5}
    \sup_{N\in\NN}\sup_{I\in\mathfrak{S}_N(\ZZ)}\norm{O_{I,N}^2(f\ast\sigma_{D^k}:k\in\ZZ)}_{L^p(\RR^d)}\lesssim_p\|f\|_{L^p(\RR^d)},\qquad f\in L^p(\RR^d),
\end{equation}
where $D>1$ is associated with the system of $D$-dyadic cubes $\{Q_{\alpha}^k\}$.
\end{theorem}
\begin{proof}
The proof is a simple repetition of the arguments given in the proof of \cite[Theorem 1.1]{jsw} but we include it for the sake of completeness. We can assume that $\int{\rm d}\sigma\neq0$ since we see that the left hand side of \eqref{eq:A5} is controlled by the square function $\big(\sum_{k\in\ZZ}|f\ast\sigma_{D^k}(x)|^2)^{1/2}$, and if $\hat{\sigma}(0)=0$ then the known estimates from \cite[Theorem A, Theorem B]{DR86} can be used to control it on $L^p(\RR^d)$. Without loss of generality assume that $\int{\rm d}\sigma=1$. Let $\phi\in C_{\rm c}^\infty(\RR^d)$ be such that $\int\phi=1$. Then one may write the following decomposition $\sigma=\phi\ast\sigma+(\delta_0-\phi)\ast\sigma,$ where $\delta_0$ is the Dirac measure at 0. Therefore we may write $f\ast\sigma_{D^k}(x)=\calL_k f(x)+\calH_k f(x),$ where
\begin{equation*}
    \calL_k f(x)=f\ast(\phi\ast\sigma)_{D^k}(x)\quad\text{and}\quad \calH_k f(x)=f\ast\big((\delta_0-\phi)\ast\sigma\big)_{D^k}(x).
\end{equation*}
Hence, by the triangle inequality it is enough to prove
\begin{align}\label{eq:A6}
    \sup_{N\in\NN}\sup_{I\in\mathfrak{S}_N(\ZZ)}\norm{O_{I,N}^2(\calL_k f:k\in\ZZ)}_{L^p(\RR^d)}\lesssim_p\|f\|_{L^p(\RR^d)}
\end{align}
and
\begin{equation}\label{eq:A7}
    \sup_{N\in\NN}\sup_{I\in\mathfrak{S}_N(\ZZ)}\norm{O_{I,N}^2(\calH_k f:k\in\ZZ)}_{L^p(\RR^d)}\lesssim_p\|f\|_{L^p(\RR^d)}.
\end{equation}
At first we handle the estimate \eqref{eq:A7}. Since the oscillation seminorm is controlled by the square function it is enough to prove
\begin{equation}\label{eq:A8}
    \norm[\big]{\big(\sum_{k\in\ZZ}|\calH_kf|^2\big)}_{L^p(\RR^d)}\lesssim_p\|f\|_{L^p(\RR^d)}.
\end{equation}
We note that $(\delta_0-\phi)\ast\sigma$ is a compactly supported measure with the vanishing mean value which satisfies condition \eqref{eq:ap1}. Hence by the known results from \cite[Theorem A, Theorem B]{DR86} we see that \eqref{eq:A8} holds.
The estimates for the low frequency part $\calL_k f$ will follow from the martingale estimates. Let $(\mathbb{E}_k)_{k\in\ZZ}$ be the dyadic martingale sequence associated with the dyadic cubes related to the dilation $t^A$. Then we may write
\begin{align*}
    \sup_{N\in\NN}\sup_{I\in\mathfrak S(\ZZ)}\norm{O_{I,N}^2(\mathcal{L}_k f:k\in\ZZ)}_{L^p(\RR^d)}\lesssim& \sup_{N\in\NN}\sup_{I\in\mathfrak S(\ZZ)}\norm{O_{I,N}^2(\calD_k f:k\in\ZZ)}_{L^p(\RR^d)}\\
    &+\sup_{N\in\NN}\sup_{I\in\mathfrak S(\ZZ)}\norm{O_{I,N}^2(\mathbb{E}_k f:k\in\ZZ)}_{L^p(\RR^d)},
\end{align*}
where $$\calD_k f(x):=f\ast(\phi\ast\sigma)_{D^k}(x)-\EE_k f(x).$$ 
Since by Theorem~\ref{thm:martinoscilation} we know that the oscillation inequality holds for the martingale $\EE_k f$, we only need to handle the part with $\calD_k f$. As before, we are reduced to show that for any $p\in(1,\infty)$ one has
\begin{equation*}
    \norm[\Big]{\big(\sum_{k\in\ZZ}|\calD_k f|^2\big)^{1/2}}_{L^p(\RR^d)}\lesssim_p\|f\|_{L^p(\RR^d)},
\end{equation*}
which follows by Lemma~\ref{martinapprox} since $\phi\ast\sigma$ is a Schwartz function such that $\int\phi\ast\sigma=1$.
\end{proof}
\subsection{Oscillation inequality for the operator $\calH_t$}
Now we are ready to give a proof of \eqref{eq:ap3}. Due to the differential nature of the oscillation seminorm it is enough to prove \eqref{eq:ap3} for the ``complement'' Radon transform given by
\begin{equation*}
    \Tilde{\calH}_t f(x):=\int_{\Omega_t^c}f(x-(y)^\Gamma)K(y){\rm d}y.
\end{equation*}
As usual in this context we divide our seminorm into long oscillations and short variations
\begin{align*}
    \sup_{N\in\NN}\sup_{I\in\mathfrak S_N(\RR_+)}\norm[\big]{O_{I,N}^2(\Tilde{\calH}_t f:t\in\RR_+)}_{L^p(\RR^\Gamma)}\lesssim& \sup_{N\in\NN}\sup_{I\in\mathfrak S_N(\ZZ)}\norm[\big]{O_{I,N}^2(\Tilde{\calH}_{D^k} f:k\in\ZZ)}_{L^p(\RR^\Gamma)}\\
    &+\norm[\Big]{\Big(\sum_{k\in\ZZ}V^2\big(\Tilde{\calH}_t f:t\in[D^k,D^{k+1})\big)^2\Big)^{1/2}}_{L^p(\RR^\Gamma)}.
\end{align*}
The estimate for the short variations follows the same lines as in \cite[Section 9.3]{MST2}, since the Littlewood--Paley theory is still available in the context of $D$-dyadic numbers -- we omit the details. We focus on proving estimates for the long oscillations. The presented approach is well known (see \cite{DR86,jsw,MSZ2}) hence we will not go into precise details. At the beginning, we see that we can express $\Tilde{\calH}_{D^k}$ as a telescoping sum
\begin{equation*}
    \Tilde{\calH}_{D^k} f(x):=\sum_{j\ge k}\mu_{D^j}\ast f(x),\qquad\text{with}\qquad \mu_{D^k}\ast f(x):=\int_{\Omega_{D^{k+1}}\setminus\Omega_{D^k}}f(x-(y)^\Gamma)K(y){\rm d}y.
\end{equation*}
Now, let $\varphi$ be a smooth compactly supported function such that $\hat{\varphi}(0)=1$. Then we employ the following decomposition (cf. \cite[Theorem E]{DR86})
\begin{align}\label{decompostionRD}
   \Tilde{\calH}_{D^k} f=\varphi_{D^k}\ast\calH f-\big(\varphi\ast\sum_{j<0}\mu_{D^j}\big)_{D^k}\ast f+\big(\sum_{j\ge0}(\delta_0-\varphi)\ast\mu_{D^{j}}\big)_{D^k}\ast f,
\end{align}
where
\begin{align*}
    \calH f(x):={\rm p.v.}\int_{\RR^k}f(x-(y)^\Gamma)K(y){\rm d}y.
\end{align*}
The oscillation inequality for the term $\varphi_{D^k}\ast\calH f$ follows by Theorem~\ref{oscillationmeasure} and by the fact that $\calH$ is a bounded operator on $L^p$. For the second term in \eqref{decompostionRD} it is enough to show
\begin{equation*}
\norm[\Big]{\Big(\sum_{k\in\ZZ}\big|\big(\varphi\ast\sum_{j<0}\mu_{D^j}\big)_{D^k}\ast f|^2\Big)^{1/2}}_{L^p(\RR^\Gamma)}\lesssim\|f\|_{L^p(\RR^\Gamma)}.
\end{equation*}
To prove the above estimate we note that $\varphi\ast\sum_{j<0}\mu_{D^j}$ is a Schwartz function with mean value zero and use the results from \cite[Theorem A, Theorem B]{DR86}. To estimate the third term occurring in \eqref{decompostionRD} it is enough to show that for some  positive constant $c_p$ one has
\begin{equation}\label{eq:A11}
    \norm[\Big]{\Big(\sum_{k\in\ZZ}\big| f\ast\big((\delta_0-\varphi)\ast\mu_{D^j}\big)_{D^k}\big|^2\Big)^{1/2}}_{L^p(\RR^\Gamma)}\lesssim D^{-c_p j}\|f\|_{L^p(\RR^\Gamma)}.
\end{equation}
The above estimate can be proved by using the Littlewood–Paley theory, for details see \cite[p. 21]{jsw}. This ends the proof of Theorem~\ref{Thm2} for the canonical mapping $(y)^\Gamma$.

\end{document}